\documentclass{artjlt}

\usepackage{amsmath,amsfonts,amssymb}

\usepackage{stmaryrd}
\usepackage{bbm}
\usepackage{color}
\usepackage[all, cmtip]{xy}

\renewcommand{\firstheadline}{\begin{minipage}[t]{9cm}
    \end{minipage}\hfill
      }

\newcommand {\emptycomment}[1]{}
\newcommand{\lon }{\,\rightarrow\,}
\newcommand{\be }{\begin{equation}}
\newcommand{\ee }{\end{equation}}
\newcommand{\defbe}{\triangleq}

\newcommand{\LieDerivation}{\mathfrak L}

\newcommand{\Real}{\mathbb R}

\newcommand{\CWM}{C^{\infty}(M)}

\newcommand{\set}[1]{\left\{#1\right\}}

\newcommand{\frkd}{\mathfrak d}

\newcommand{\frkj}{\mathfrak j}

\newcommand{\frkr}{\mathfrak r}

\newcommand{\frky}{\mathfrak y}

\newcommand{\frkD}{\mathfrak D}

\newcommand{\frkJ}{\mathfrak J}

\newcommand{\frkL}{\mathfrak L}

\newcommand{\half}{\frac{1}{2}}

\newcommand{\conpairing}[1]{\left\langle  #1\right\rangle }

\newcommand{\Dorfman}[1]{\{ #1\}}

\newcommand{\jet}{\mathfrak{J}}

\newcommand{\jetd}{\mathbbm{d}}
\newcommand{\dev}{\mathfrak{D}}

\newcommand{\Id}{\mathrm{Id}}

\newcommand{\e}{\mathbbm{e}}
\newcommand{\p}{\mathbbm{p}}
\newcommand{\id}{\mathbbm{i}}
\newcommand{\jd}{\mathbbm{j}}

\newcommand{\omni}{\mathcal{E}}

\newcommand{\Hom}{\mathrm{Hom}}

\newcommand{\gl}{\mathfrak {gl}}

\title{Higher omni-Lie algebroids}                                     

\author{Yanhui Bi
\footnote{Research partially supported by NSFC (11601219), and Jiangxi' Natural Science Foundation (20161BAB211003)},
Luca Vitagliano
\footnote{LV is member of the GNSAGA of INdAM},
Tao Zhang
\footnote{Corresponding author. Research partially supported by NSFC (11501179)}}                 
\lastname{Bi, Vitagliano, Zhang}  



\address{%
Yanhui Bi\\               
College of Mathematics and Information Science, Nanchang Hangkong University, Nanchang 330063, PR China\\            
biyanhui0523@gmail.com                
}

\address{%
Luca Vitagliano\\               
DipMat, Universit\`a degli Studi di Salerno, via Giovanni Paolo II n${}^\circ$ 123, 84084 Fisciano (SA) Italy\\            
lvitagliano@unisa.it               
}

\address{%
Tao Zhang\\               
College of Mathematics and Information Science, Henan Normal University, Xinxiang 453007, PR China\\            
zhangtao@htu.edu.cn               
}

\begin{document}


\maketitle

\begin{abstract}
We propose a definition of a ``higher'' version of the omni-Lie algebroid and study its isotropic and involutive subbundles. Our higher omni-Lie algebroid is to (multi)contact and related geometries what the higher generalized tangent bundle of Zambon and Bi/Sheng is to (multi)symplectic and related geometries.   
\end{abstract}

\section{Introduction}

The \emph{generalized tangent bundle} of a smooth manifold is the direct sum $TM \oplus T^\ast M$. It is the natural arena for various interesting geometries, e.g.~Dirac geometry, and generalized complex geometry. The generalized tangent bundle is an instance of a \emph{Courant algebroid}. The \emph{omni-Lie algebroid} of a vector bundle $E \to M$ is the direct sum $\mathfrak D E \oplus \mathfrak J E$, where $\mathfrak D E$ is the \emph{gauge algebroid}, and $\mathfrak J E$ is the first jet bundle of $E$. The omni-Lie algebroid of a line bundle is the natural arena for Dirac-Jacobi geometry, and generalized complex geometry in odd dimensions \cite{Wade2000, Iglesias2004, Vitagliano, Vitagliano3, Vitagliano4}, and it is an instance of an $E$-Courant algebroid. Now, fix a positive integer $n$. There is an ``$n$-form version'' of the generalized tangent bundle, called the \emph{higher generalized tangent bundle}: $TM \oplus \wedge^n T^\ast M$. The higher generalized tangent bundle and its isotropic and involutive subbundles have been first considered in \cite{Bi2011,Zambon2012} (see also \cite{Y.Bi2015,BAR2016}). They encode higher Dirac structures, in particular closed $n$-forms and certain Nambu structures \cite{Zambon2012}.

It is then natural to ask: \emph{Is there an ``$n$-form version'' of the omni-Lie algebroid? And, if yes, what kind of geometries do its isotropic and involutive subbundles encode?} In this paper, we (partially) answer these questions.
We begin with what we call \emph{vector bundle forms} (see Subsection \ref{Sec:vect_form}). They are sections of a certain graded subbundle $\jet_{\bullet} E$ of the bundle $\Hom(\wedge^\bullet \dev E, E)$ of $E$-valued forms on the gauge algebroid. Vector bundle forms are for a vector bundle $E \to M$ what ordinary differential forms are for the trivial line bundle $M \times \mathbb R \to M$. With vector bundle forms at hand, we are able to define the \emph{higher omni-Lie algebroid} (see Subsection \ref{Sec:higher_omni}) as the direct sum $\mathfrak D E \oplus \jet_{n} E$ (for any fixed $n$). In the same way as vector bundle forms, higher omni-Lie algebroids are for a vector bundle what the higher generalized tangent bundles are for the trivial line bundle. Various interesting ``higher geometric structures'' appear as ``Dirac structures'' in the higher omni-Lie algebroid. For instance, it turns out that certain ``higher contact structures'' \cite{Vitagliano2} can be seen as isotropic involutive subbundles of the higher omni-Lie algebroid.

The paper is organized as follows. In Section 2, we introduce the concept of vector bundle forms and the higher omni-Lie algebroid of a vector bundle.
In section 3, we study isotropic and involutive subbundles of a higher omni-Lie algebroid which project isomorphically either on $\mathfrak D E$ or on $\jet_{n} E$.
In the last Section 4, we study higher omni-Lie algebroids for the spacial case when $E$ is the trivial line bundle.

 \section{Higher omni-Lie algebroids}

\subsection{Preliminaries: omni-Lie algebroids}

We begin recalling briefly the notion of omni-Lie algebroid  \cite{Chen-Liu}, which generalizes Weinstein's omni-Lie algebras \cite{A.Wein} to the realm of Lie algebroids. Given a vector bundle $E \to M$, let $\dev E$ be the \emph{gauge algebroid} of $E$. Recall that a section of $\dev E$ is a \emph{derivation} of $E$, i.e.~an $\mathbb R$-linear operator $\frkd : \Gamma (E) \to \Gamma (E)$ satisfying the following Leibniz rule:
\[
\frkd (f u) = \jd (\frkd) (f) u + f \frkd u, \quad \forall u \in \Gamma (E),
\]
for some, necessarily unique, vector field $\jd (\frkd) \in \mathfrak X (M)$. The gauge algebroid $\dev E$ is a Lie algebroid with Lie bracket given by the commutator of derivations, denoted $[-,-]_{\dev}$, and anchor given by the vector bundle map $\jd : \dev E \to TM$. The anchor $\jd$ is actually surjective, and its kernel $\gl (E)$ consists of vector bundle endomorphisms $\Phi : E \to E$ (covering the identity). Hence there is a short exact sequence
\begin{equation}\label{eq:SES_dev}
\xymatrix{0 \ar[r] & \gl(E)  \ar[r]^-{\id} &
                \dev{E}  \ar[r]^-{\jd} & TM \ar[r]  & 0
                }
\end{equation}
sometimes called the \emph{Atiyah sequence} of $E$. Notice that $\gl(E) $, hence $\dev E$, is equipped with a canonical section: the identity endomorphism $\Id_E : E \to E$.

Now denote by $\jet E$ the first jet bundle of $E$. Besides the usual definition, $\jet E$ can be equivalently defined from $\dev E$ as follows:
\begin{equation}
\jet E =
\set{\nu\in \Hom( \dev{E}  ,E ) :  \nu(\Phi)=\Phi\circ
\nu(\Id_E),\forall \Phi\in \gl(E )}  \subset \Hom( \dev{E}
,E ).
\end{equation}
In its turn, $\dev E$ is determined by $\jet E$:
\begin{equation}\label{eq:DE}
 {\dev E}  \cong \set{h\in \Hom( {\jet E} ,E ) : \exists
v \in T M, \mbox{ s.t. } h(\frky)=\frky(v), \forall \frky\in \Hom(TM,E) }.
\end{equation}
Given a vector bundle map $h : \jet E \to E$ as in \eqref{eq:DE}, the associated derivation $\frkd$ is given by
\[
\frkd u = h (\frkj u ), \quad u \in \Gamma (E),
\]
where $\frkj : \Gamma (E) \to \Gamma (\jet E)$ is the first jet prolongation.
The first jet bundle sits itself in a short exact sequence:
\begin{equation}\label{eq:SES_jet}
\xymatrix{0 \ar[r] & \Hom(TM,E)  \ar[r]^-{\e}
&
                {\jet}{E} \ar[r]^-{\p} & E \ar[r]  & 0,
                }
\end{equation}
where $\p : \jet E \to E$ is the usual projection, and $\e : \Hom (TM, E) \to \jet E$ is given by:
\[
\e (h) = h \circ \jd, \quad h \in \Hom (TM, E).
\]
We will often understand both the embeddings $\e$ and $\id$ when there is no risk of confusion.
We summarize the relationship between $\dev E$ and $\jet E$ with the $E$-valued, bilinear pairing
\[
\conpairing{-,-}_E : \jet E \otimes \dev E \to E, \quad (\nu, \frkd) \mapsto \nu (\frkd).
\]
Then $(\jd, \e)$ and $(\id, \p)$ are pairs of mutually adjoint maps with respect to $\conpairing{-,-}_E$.

\begin{remark}
There is yet another description of $\dev E$ and $\jet E$. Namely, given a vector bundle $E \to M$, the diagram
\begin{equation}\label{DVB}
\xymatrix{
TE \ar[r] \ar[d] & E \ar[d] \\
TM \ar[r] & M}
\end{equation}
is a double vector bundle (more precisely a VB-algebroid). Then sections of $\dev E$ identify with \emph{linear sections} of the horizontal bundle $TE \to E$, while sections of $\jet E$ identify with linear sections of the vertical bundle $TE \to TM$. Here, by \emph{linear section}, we mean those sections that are, additionally, vector bundle maps. Double vector bundle \eqref{DVB} is also equivalent to a certain vector bundle of graded manifolds. Namely, shifting by one the degree in the fibers of the horizontal bundles $TE \to E$, and $TM \to M$, we get the graded vector bundle
\begin{equation}\label{graded}
\begin{array}{c}
\xymatrix{
T[1] E \ar[d]\\
T[1] M}
\end{array},
\end{equation}
and sections of $\jet E$ do also identify with degree 0 sections of \eqref{graded}. For more details about VB-algebroids see \cite{M1998, GSM2010}. See also \cite{V2012} for the relationship between VB-algebroids and graded manifolds.
\end{remark}

The \emph{omni-Lie algebroid} of $E$ \cite{Chen-Liu, Chen-Liu-Sheng2} is the direct sum $\omni (E) = \dev E \oplus \jet E$. The main structures on $\omni (E)$ are:
\begin{itemize}
\item the projection $\rho : \omni (E) \to \dev E$ onto the first summand;
\item the (Dorfman-like) bracket $\{-,-\} : \Gamma (\omni (E)) \times \Gamma (\omni (E)) \to \Gamma (\omni (E))$ defined by:
\begin{equation}\label{eq:Dorf_1}
\{ \frkd + \mu, \frkr + \nu \} = [\frkd, \frkr]_\dev + \LieDerivation_{\frkd} \nu - \LieDerivation_{\frkr} \mu + \frkj \conpairing{\nu, \frkr}_E, \quad \forall\frkd + \mu,\, \frkr + \nu \in \Gamma (\mathcal E),
\end{equation}
\item the $E$-valued, symmetric bilinear pairing $(-,-)_+ : \omni (E) \otimes \omni (E) \to E$ defined by
\[
(\frkd + \mu, \frkr + \nu)_+ = \frac{1}{2} \left(\conpairing{\mu, \frkr}_E + \conpairing{\nu, \frkr}_E \right).
\]
\end{itemize}
Formula \eqref{eq:Dorf_1} requires some explanations. Here $\mathfrak L$ is the \emph{Lie derivative} of jets along derivations and it is uniquely defined by
\[
\conpairing{\frkL_\frkd \mu, \frkr }_E = \frkd \conpairing{\mu, \frkr}_E  - \conpairing{\mu, [\frkd, \frkr]_\dev}_E.
\]
Together with the structure maps $\rho, \{-,-\}, (-,-)_+$, the omni-Lie algebroid $\omni (E)$ is an $E$-Courant algebroid \cite{Chen-Liu-Sheng1}.

\subsection{Vector bundle forms}\label{Sec:vect_form}

 In this section, we propose an ``$n$-form version'' of the omni-Lie algebroid and study its isotropic and involutive subbundles. We begin with a ``vector bundle version'' of differential forms. In the case $n = 1$ our definition of an ``higher form'' should reproduce sections of the jet bundle $\jet E$.

The gauge algebroid $\dev{E}$ acts tautologically on the vector bundle $E$. Hence there is a natural cochain complex
\[
\left(\Omega^\bullet(\dev E, E), \jetd \right),
\]
the \emph{de Rham complex} of $\dev E$ with coefficients in its representation $E$. Here
\[
\Omega^\bullet(\dev E, E)
\triangleq \Gamma(\Hom(\wedge^{\bullet}{\dev{E}},E)),
\]
and the differential $\jetd: \Omega^k(\dev E, E)\lon \Omega^{k+1}(\dev E, E)$ is given by the usual formula:
\begin{equation}\label{coboundary}
\begin{aligned}
\jetd\mu(\frkd_{0},\frkd_{1},\cdots,\frkd_{k})  & \triangleq \sum_{i=0}^k(-1)^{i}\frkd_{i}\left(\mu(\frkd_{0},\cdots,\widehat{\frkd_{i}},\cdots,\frkd_{k})\right) \\
 & \quad +\sum_{0\leq i<j\leq k}(-1)^{i+j}\mu([\frkd_{i},\frkd_{j}],\cdots,\widehat{\frkd_{i}},\cdots,\widehat{\frkd_{j}},\cdots,\frkd_{k}),
 \end{aligned}
\end{equation}
for all $\mu \in \Omega^k(\dev E, E)$, $\frkd_i \in \Gamma (\frkD E)$, where a hat ``$\widehat{-}$'' denotes omission.

\begin{remark}\label{rem:contraction}
de Rham complex $(\Omega^\bullet(\dev E, E), \jetd )$ is actually acyclic \cite{Rubtsov}. Even more, it possesses a canonical contracting homotopy given by contraction $\iota_{\Id_E}$ with the identity endomorphism $\Id_E$.
\end{remark}

Clearly, $\Omega^\bullet(\dev E, E)$ is a DG-module over $\Omega^\bullet(\dev E) \triangleq \Gamma (\Hom (\wedge^\bullet \dev E, \Real_M))$. Using $\jd : \dev E \to TM$, we can define a DG-algebra map
\[
\jd^\ast : \Omega^\bullet (M) \to \Omega^\bullet(\dev E), \quad \omega \mapsto \omega (\jd -, \ldots, \jd -).
\]
Finally we can \emph{change the scalars} via $\jd^\ast$, and give $\Omega^\bullet(\dev E, E)$ the structure of a DG-module over $\Omega^\bullet (M)$.
Notice that, for all $u\in\Gamma (E)$, $\jetd u = \frkj u \in \Omega^1(\dev E,
E)$  is actually a section of $\jet E$ and, from the $\Omega^\bullet (M)$-DG-module property, we have that
\[
\jetd(fu)=f\jetd u+ d f\otimes u,\quad \forall f\in \CWM,~ u\in
\Gamma(E).
\]
Now we want to extend $\Gamma (\jet E)$ to a whole $\Omega^\bullet (M)$-DG-submodule $\Omega^\bullet_{\jet E} $ of $\Omega^\bullet(\dev E, E)$. The DG-module $\Omega^\bullet_{\jet E}$ will be our vector bundle version of differential forms.

In the graded vector bundle $\Hom(\wedge^\bullet \dev E, E)$ consider the graded subbundle \cite{Chen-Liu-Sheng1}:
\begin{equation}\label{eq:J_E}
\begin{aligned}
\jet_\bullet E \triangleq \{\mu \in
\Hom(\wedge^\bullet \dev E, E) & : \exists \lambda_{\mu}\in
\Hom(\wedge^{\bullet-1}TM, E), \\
& \mbox{ s.t. }
 \iota_\Phi \mu =\Phi\circ \jd^\ast {\lambda_{\mu}},~ \forall\Phi\in \gl(E)\}.
 \end{aligned}
\end{equation}
Sections of $\jet_\bullet E$ will be denoted shortly by $\Omega^\bullet_{\jet E}$, and the degree $n$ homogeneous component will be denoted $\jet_n E$. In particular $\Omega^n_{\jet E} = \Gamma (\jet_n E)$. Notice that $\jet_1 E = \jet E$. The space $\jet_\bullet E$ and its section were first considered in \cite{Chen-Liu-Sheng1} (under a different notation).

\begin{remark}
Let $\mu \in \jet_\bullet E$. Then the $E$-valued form $\lambda_\mu$ is necessarily unique and it is completely determined by the condition
\begin{equation}\label{eq:lambda}
\jd^\ast \lambda_\mu = \lambda_\mu (\jd -, \ldots, \jd -) = \iota_{\Id_E} \mu,\quad \text{for all $\Phi \in \gl (E)$.}
\end{equation}

\end{remark}

\begin{proposition}\label{prop:VB-forms}
$\Omega^\bullet_{\jet E}$ is an $\Omega^\bullet (M)$-submodule in $\Omega^\bullet (\dev E, E)$. Additionally, it is preserved, for all $\frkd \in \Gamma (\dev E)$, by\\
(1) the de Rham differential $\jetd$,
(2) contractions $\iota_\frkd$,
(3) Lie derivatives $\LieDerivation_\frkd$,
\end{proposition}

\begin{proof}
First we show that $\Omega^\bullet_{\jet E}$ is an $\Omega^\bullet (M)$-submodule. Let $\omega \in \Omega^k (M)$, and $\mu \in \Omega^l (\dev E, E)$. Then the product $\omega \wedge \mu$ is given by
\begin{equation}\label{eq:dot}
\omega \wedge \mu (\frkd_1, \ldots, \frkd_{k+l}) = \sum_{\sigma \in S_{k, l}} (-1)^{\sigma} \omega(\jd (\frkd_{\sigma(1)}), \ldots, \jd(\frkd_{\sigma(k)})) \mu (\frkd_{\sigma_{k+1}}, \ldots, \frkd_{\sigma{k+l}}),
\end{equation}
for all $\frkd_i \in \Gamma (\dev E)$, where $S_{k,l}$ denoted $(k,l)$-unshuffles.
Now let $\mu \in \Omega^l_{\jet E}$, and let $\Phi \in \Gamma (\gl(E))$. As $\jd (\Phi) = 0$, from  \eqref{eq:dot}, we immediately get
\[
\iota_\Phi \omega \wedge \mu = (-1)^k \omega \wedge \iota_\Phi \mu = (-1)^k \omega \wedge \left( \Phi \circ \jd^\ast \lambda_\mu \right) = \Phi \circ \jd^\ast ((-1)^k  \omega \wedge \lambda_\mu).
\]
This shows that $\omega \wedge \mu \in \Omega^{k+l}_{\jet E}$ and $\lambda_{\omega \wedge \mu} = (-1)^k \omega \wedge \lambda_\mu$.

That $\Omega^\bullet_{\jet E}$ is preserved by the de Rham differential $\jetd$, and Lie derivatives $\LieDerivation_\frkd$, is proved in \cite{Chen-Liu-Sheng1}. To see that it is also preserved by contractions $\iota_\frkd$, let $\mu, \Phi$ be as above, and compute
\[
\iota_\Phi \iota_\frkd \mu = - \iota_\frkd \iota_\Phi \mu = -\iota_\frkd (\Phi \circ \jd^\ast \lambda_\mu) = - \Phi \circ \jd^\ast \iota_{\jd(\frkd)} \lambda_\mu.
\]
This shows that $\iota_\frkd \mu \in \Omega^{l-1}_{\jet E}$, and $\lambda_{\iota_\frkd \mu}= - \iota_{\jd(\frkd)} \lambda_\mu$.
\end{proof}

\begin{remark}\label{rem:Cartan}
From Proposition \ref{prop:VB-forms}, the full \emph{Cartan calculus} restricts to $\Omega^\bullet_{\jet E}$. Additionally, from Remark \ref{rem:contraction} and Proposition \ref{prop:VB-forms} (last point), $(\Omega^\bullet_{\jet E}, \jetd)$ is an acyclic subcomplex and $\iota_{\Id_E}$ is a contracting homotopy for it.
\end{remark}

\begin{proposition}
The short exact sequence \eqref{eq:SES_jet} extends to a (degree-wise) short exact sequence of $\wedge^\bullet T^\ast M$-modules
\begin{equation}\label{eq:SES_higher}
\xymatrix{0 \ar[r] & \wedge^\bullet T^\ast M \otimes E  \ar[r]^-{\e_\bullet}
&
                {\jet}_\bullet{E} \ar[r]^-{\p_\bullet} & \wedge^{\bullet} T^\ast M \otimes E\ar[r]  & 0
                },
\end{equation}
where
\[
\p_\bullet (\mu) = \lambda_\mu, \quad \text{and} \quad \e_\bullet (\lambda') = \jd^\ast \lambda'.
\]
\end{proposition}
\begin{proof}
See \cite{Chen-Liu-Sheng1}.
\end{proof}

\begin{remark}\label{rem:split}
The sequence \eqref{eq:SES_higher} does not split (in the category of $\wedge^\bullet T^\ast M$-modules) in general. However the induced sequence in sections:
\begin{equation}\label{eq:SES_Omega}
\xymatrix{0 \ar[r] & \Omega^\bullet (M, E)  \ar[r]^-{\e_\bullet}
&
              \Omega^\bullet_{\jet E} \ar[r]^-{\p_\bullet} & \Omega^{\bullet-1} (M, E)\ar[r]  & 0
                },
\end{equation}
splits canonically in the category of graded vector spaces. Specifically, the map
\[
\frkj_\bullet : \Omega^{\bullet-1} (M, E) \to \Omega^\bullet_{\jet E}, \quad \lambda \mapsto \jetd \jd^\ast \lambda
\]
is a right splitting. As a consequence, there is a canonical isomorphism of graded vector spaces
\begin{equation}\label{eq:iso}
\Omega^\bullet_{\jet E} \cong \Omega^\bullet (M, E) \oplus \Omega^{\bullet-1}(M, E)
\end{equation}
identifying $\mu \in \Omega^k_{\jet E}$ with a pair $(\mu_0, \mu_1)$ consisting of an $E$-valued $k$-form, and a $E$-valued $(k-1)$-form. It is easy to see that, actually,
\begin{equation}\label{eq:0_1}
(\mu_0, \mu_1) = (\lambda_{\jetd \mu}, \lambda_\mu).
\end{equation}
It then follows from Remark \ref{rem:Cartan} that:
\begin{equation}\label{eq:dj+j}
\mu =  \jd^\ast \mu_0 + \jetd \jd^\ast \mu_1.
\end{equation}
Notice that $\frkj_1 = \frkj$ is just the first jet prolongation. Finally, we describe all natural operations on $\Omega^\bullet_{\jet E}$ in terms of the isomorphism \eqref{eq:iso}. So, let $\omega \in \Omega^\bullet (M)$, and let $\frkd \in \Gamma (\dev E)$. The multiplication by $\omega$, the de Rham differential, the contraction and the Lie derivative with $\frkd$ induce operations on $\Omega^\bullet (M, E) \oplus \Omega^{\bullet-1}(M, E)$ which we still denote by $\omega \wedge -$, $\jetd$, $\iota_\frkd$, $\LieDerivation_\frkd$. A direct computation exploiting either \eqref{eq:0_1} or \eqref{eq:dj+j} then shows that, for all $\mu \in \Omega^\bullet_{\jet E}$,
\begin{align}
\omega \wedge (\mu_0, \mu_1) & = \left(\omega \wedge \mu_0 - (-1)^\omega d\omega \wedge \mu_1 , (-1)^\omega \omega \wedge \mu_1 \right), \nonumber \\
\jetd (\mu_0, \mu_1) & = \left( 0, \mu_0 \right), \label{eq:d}\\
\iota_\frkd (\mu_0, \mu_1) & = \left( \iota_{\jd (\frkd)} \mu_0+ \LieDerivation_\frkd \mu_1, - \iota_{\jd (\frkd)} \mu_1\right) \label{eq:iota}\\
\LieDerivation_\frkd (\mu_0, \mu_1) & = \left( \LieDerivation_\frkd \mu_0, \LieDerivation_\frkd \mu_1 \right).  \label{eq:Lie}
\end{align}
In particular
\[
\iota_{\Id_E} (\mu_0, \mu_1) = (\mu_1, 0).
\]
In Formulas \eqref{eq:iota} and \eqref{eq:Lie} it appears the Lie derivative along $\frkd \in \Gamma (\dev E)$ of an $E$-valued form $\nu \in \Omega^\bullet (M, E)$. This requires an explanation. Let $\nu$ be of degree $k$, then $\LieDerivation_\frkd \nu$ is the $E$-valued $k$-form on $M$ given by
\[
\LieDerivation_\frkd \nu (X_1, \ldots, X_k) = \frkd (\nu (X_1, \ldots, X_k)) - \sum_{i=1}^k \nu (X_1, \ldots, [\jd(\frkd), X_i], \ldots, X_k),
\]
for all $X_i \in \mathfrak X (M)$.
\end{remark}

\begin{remark}
Notice that $\Omega^\bullet (M)$ can be seen as the graded algebra of function on the graded manifold $T[1] M$. It is not hard to see, e.g. in local coordinates, that $\Omega^\bullet_{\frkJ E}$ is also the $\Omega^\bullet (M)$-module of sections of the graded bundle (\ref{graded}).
\end{remark}

\subsection{The higher omni-Lie algebroid}\label{Sec:higher_omni}

We are now ready to give a definition of \emph{higher omni-Lie algebroid}. Let $E \to M$ be a vector bundle, and let $n$ be a positive integer. We begin noticing that there exists a $\jet_{n-1} E$-valued, bilinear pairing
\[
\langle -, -\rangle_{\jet E} : \jet_n E \otimes \dev E \to \jet_{n-1} E, \quad (\mu, \frkd) \mapsto \iota_\frkd \mu.
\]

\begin{Definition}\label{higheromni}
The \emph{$n$-omni-Lie algebroid} or, simply, the \emph{higher omni-Lie algebroid} of $E$, is the quadruple $(\omni_n (E),\rho,  \Dorfman{-,-}, (-,-)_+)$, where $\omni_n(E)=\dev{E}\oplus \jet_n E$. Additionally
\begin{itemize}
\item $\rho : \omni_n (E) \to \dev E$ is the projection onto the first summand,
\item  $\Dorfman{-,-}:\Gamma(\omni_n (E))\times\Gamma(\omni_n (E))\longrightarrow\Gamma(\omni_n (E))$
is the bracket defined by:
\begin{equation}\label{Dorfman}
\Dorfman{\frkd+\mu,\frkr+\nu}\defbe[\frkd,\frkr]_{\dev}+\LieDerivation_{\frkd}\nu-\LieDerivation_{\frkr}\mu
+\jetd\conpairing{\mu,\frkr}_{\jet E} = [\frkd,\frkr]_{\dev}+\LieDerivation_{\frkd}\nu- \iota_{\frkr} \jetd \mu,
\end{equation}
for all $\frkd,\frkr\in \Gamma (\dev{E})$, and $\mu,\nu\in \Gamma(\jet_n E)$, and
\item $({-,-})_+ : \omni_n (E) \otimes \omni_n(E) \to \jet_{n-1} E$ is the symmetric bilinear pairing defined by:
\begin{equation}\label{pair}
({\frkd+\mu,\frkr+\nu})_+ \defbe
\half \left(\conpairing{\frkd,\nu}_{\jet E}  +\conpairing{\frkr,\mu}_{\jet E} \right),
\end{equation}
for all $\frkd,\frkr\in\dev{E}$, and $\mu,\nu\in\jet_n E$.
\end{itemize}
The bracket $\Dorfman{-,-}$ is called the \emph{higher Dorfman bracket}.
\end{Definition}

Notice that the $1$-omni-Lie algebroid $\omni_1 (E) = \omni (E)$ is just the omni-Lie algebroid. In the following, we will often denote $\omni_n(E)$ simply by $\omni$ if this does not lead to confusion.

%
%

\begin{example}\label{example:linebundle}
Let $E = \ell$ be a line bundle. Then every first order differential operator $\Gamma (\ell) \to \Gamma (\ell)$ is a derivation. Hence $\dev \ell \cong \Hom (\jet \ell, \ell)$ and $\jet \ell \cong \Hom (\dev \ell, \ell)$. Similarly, the condition on $\lambda$ in \eqref{eq:J_E} is empty and we have
\[
\mathfrak{J}_\bullet \ell = \Hom(\wedge^\bullet \mathfrak{D}\ell,\ell).
\]
It follows that $\Omega^\bullet_{\jet \ell} = \Omega^\bullet (\mathfrak{D}\ell,\ell)$ and
\[
\omni_n (\ell) = \dev \ell \oplus \Hom (\wedge^n \dev \ell, \ell).
\]
\end{example}



\begin{theorem}\label{Thm:Property of omni}
 Let $(\omni, \rho, \Dorfman{-,-}, (-,-)_+)$ be the $n$-omni-Lie algebroid of a vector bundle $E \to M$. Then
\begin{itemize}
\item[\rm(i)] $(\Gamma(\omni),\Dorfman{-,-})$ is a Leibniz algebra;
\item[\rm(ii)] $\rho(\Dorfman{e_1,e_2})=[\rho(e_1),\rho(e_2)]_{\dev}$,
\item[\rm(iii)]
$\Dorfman{e_1,fe_2}=f\Dorfman{e_1,e_2}+\conpairing{\jd\circ\rho(e_1),d f}(e_2)$,
\item[\rm(iv)]$\LieDerivation_{\rho(e_1)}({e_2,e_3})_+=(\Dorfman{e_1,e_2},e_3)_++(e_2,\Dorfman{e_1,e_3})_+$,
\item[\rm(v)] $\Dorfman{e,e}= \jetd ({e,e})_+$,
\end{itemize}
for all $e, e_1, e_2, e_3\in \Gamma(\omni)$ and $f\in\CWM$.
\end{theorem}
\begin{proof} Similarly as in the case of the generalized tangent bundle and the omni-Lie algebroid, the statement follows from standard \emph{Cartan calculus} on $\Omega^\bullet_{\jet E}$. We report here some details of the proof for completeness.

(i). Write $e_i=\frkd_i+\mu_i$,  where $\frkd_i\in\Gamma(\dev{E})$ and $\mu_i\in\Gamma(\jet_n E)$, $i=1,2,3$. Then we have
\begin{equation*}
\begin{aligned}
 \Dorfman{\frkd_1+\mu_1,\Dorfman{\frkd_2+\mu_2,\frkd_3+\mu_3}}
=&\,[\frkd_1,[\frkd_2,\frkd_3]_{\dev}]_{\dev}+\LieDerivation_{\frkd_1}\LieDerivation_{\frkd_2}\nu_3
  -\LieDerivation_{\frkd_1}\LieDerivation_{\frkd_3}\nu_2\\
  &+\LieDerivation_{\frkd_1}\jetd i_{\frkd_3}\nu_2
 -\LieDerivation_{[\frkd_2,\frkd_3]_{\dev}}\nu_1+\jetd i_{[\frkd_2,\frkd_3]_{\dev}}\nu_1,
 \end{aligned}
 \end{equation*}
 \begin{equation*}
 \begin{aligned}
 \Dorfman{\Dorfman{\frkd_1+\mu_1,\frkd_2+\mu_2},\frkd_3+\mu_3}
=&\,[[\frkd_1,\frkd_2]_{\dev},\frkd_3]_{\dev}
 +\LieDerivation_{[\frkd_1,\frkd_2]_{\dev}}\nu_3-\LieDerivation_{\frkd_3}\LieDerivation_{\frkd_1}\nu_2\\
& +\LieDerivation_{\frkd_3}\LieDerivation_{\frkd_2}\nu_1
-\LieDerivation_{\frkd_3}\jetd i_{\frkd_2}\nu_1+\jetd i_{\frkd_3}\LieDerivation_{\frkd_1}\nu_2 \\
&-\jetd i_{\frkd_3}\LieDerivation_{\frkd_2}\nu_1+\jetd i_{\frkd_3}\jetd i_{\frkd_2}\nu_1,
\end{aligned}
 \end{equation*}
 \begin{equation*}
 \begin{aligned}
 \Dorfman{\frkd_2+\mu_2,\Dorfman{\frkd_1+\mu_1,\frkd_3+\mu_3}} =&\,[\frkd_2,[\frkd_1,\frkd_3]_{\dev}]_{\dev}+\LieDerivation_{\frkd_2}\LieDerivation_{\frkd_1}\nu_3
  -\LieDerivation_{\frkd_2}\LieDerivation_{\frkd_3}\nu_1\\
&+\LieDerivation_{\frkd_2}\jetd i_{\frkd_3}\nu_1   -\LieDerivation_{[\frkd_1,\frkd_3]_{\dev}}\nu_2+\jetd i_{[\frkd_1,\frkd_3]_{\dev}}\nu_2.
\end{aligned}
\end{equation*}
So, it is enough to show that
\begin{equation*}
\begin{aligned}
 \LieDerivation_{\frkd_1}\jetd i_{\frkd_3}\nu_2+\jetd i_{[\frkd_2,\frkd_3]_{\dev}}\nu_1
 =&\, \LieDerivation_{\frkd_2}\jetd i_{\frkd_3}\nu_1+\jetd i_{[\frkd_1,\frkd_3]_{\dev}}\nu_2  -\LieDerivation_{\frkd_3}\jetd i_{\frkd_2}\nu_1\\
&+\jetd i_{\frkd_3}\LieDerivation_{\frkd_1}\nu_2  -\jetd i_{\frkd_3}\LieDerivation_{\frkd_2}\nu_1+\jetd i_{\frkd_3}\jetd i_{\frkd_2}\nu_1.
 \end{aligned}
 \end{equation*}
But this identity follows straightforwardly from the following \emph{Cartan formulas} (see, e.g., \cite{Chen-Liu})
 \begin{equation}\label{Cartan}
[\iota_\frkd, \jetd] = \LieDerivation_\frkd, \quad [\LieDerivation_\frkd, \jetd] = 0, \quad [\iota_{\frkd_1}, \LieDerivation_{\frkd_2}] = \iota_{[\frkd_1, \frkd_2]},
 \end{equation}
 where $[-,-]$ is the graded commutator.

 (ii). It follows from $\rho(\Dorfman{\frkd_1+\mu_1,\frkd_2+\mu_2})=[\frkd_1,\frkd_2]_{\dev}$.

(iii). We have
\begin{equation*}
\begin{aligned}
  \Dorfman{\frkd_1+\mu_1,f(\frkd_2+\mu_2)}&=[\frkd_1,f\frkd_2]_{\dev}+\LieDerivation_{\frkd_1}(f\mu_2)-\iota_{f\frkd_2}\jetd\mu_1\\
  &=f\Dorfman{\frkd_1+\mu_1,\frkd_2+\mu_2}+\conpairing{\frkd_1,\jetd f}(\frkd_2+\mu_2)\\
  &= f\Dorfman{\frkd_1+\mu_1,\frkd_2+\mu_2}+\conpairing{\rho(\frkd_1+\mu_1),\jetd f}(\frkd_2+\mu_2).
  \end{aligned}
\end{equation*}
This concludes the proof of (iii).

(iv). The right side of (iv) is
\[
\frac{1}{2}(\iota_{[\frkd_1,\frkd_2]_{\dev}}\mu_3+\iota_{\frkd_3}\LieDerivation_{\frkd_1}\mu_2
+\iota_{\frkd_2}\LieDerivation_{\frkd_1}\mu_3+\iota_{[\frkd_1,\frkd_3]_{\dev}}\mu_2)
\]
and the left hand side is
\[
\frac{1}{2}(\LieDerivation_{\frkd_1}i_{\frkd_2}\mu_3+\LieDerivation_{\frkd_1}i_{\frkd_3}\mu_2).
\]
Now the statement follows from the third one of \eqref{Cartan}.

(v). It immediately follows from \eqref{Dorfman} and \eqref{pair}.
\end{proof}

\begin{remark}
The properties of the omni-Lie algebroid can be axiomatized and this produces the definition of an $E$-Courant algebroid \cite{Chen-Liu-Sheng1}. Similarly it should be possible to axiomatize the properties of the higher omni-Lie algebroid listed in Theorem \ref{Thm:Property of omni} and produce a definition of \emph{higher $E$-Courant algebroid}. Specifically, let $E \to M$ be a vector bundle. Then a \emph{higher $E$-Courant algebroid} should be a quadruple $(\omni, \rho, \{-,-\}, (-,-)_+)$ where
\begin{itemize}
\item $\omni \to M$ is a vector bundle,
\item $\rho : \omni \to \dev E$ is a vector bundle map,
\item $\{-,-\} : \Gamma (\omni) \times \Gamma (\omni) \to \Gamma (\omni)$ is an $\mathbb R$-bilinear bracket, and
\item $(-,-)_+ : \omni \otimes \omni \to \jet_{n-1} E$ is a non-degenerate, symmetric, bilinear pairing.
\end{itemize}
Additionally, the structure maps $\rho, \{-,-\}, (-,-)_+$ should satisfy the properties (i)--(iv) in the statement of Theorem \ref{Thm:Property of omni}. Finally, we speculate that property (v) in Theorem \ref{Thm:Property of omni} should be replaced by the following identity
\[
(\Dorfman{e,e}, e')_+= \iota_{\rho(e')} \jetd (e,e)_+,
\]
for all $e, e' \in \Gamma(\omni)$. Investigating this definition goes beyond the scopes of this paper. Hopefully, we will follow this line of thoughts elsewhere.
\end{remark}

\begin{remark}
The Dorfman bracket on sections of the higher generalized tangent bundle $TM \oplus \wedge^n T^\ast M$ can be deformed by a closed $(n+2)$-form \cite{Bi2011}. Similarly, the higher Dorfman bracket can be deformed by a section of $\jet_{n+2} E$. To see this, take $\omega \in \Omega^{n+2}_{\jet E}$ and define a new \emph{deformed higher Dorfman bracket}
\[
\Dorfman{-,-}_{\omega} : \Gamma (\omni) \times \Gamma (\omni) \to \Gamma (\omni)
\]
 by
\begin{equation}\label{Twisting Bracket}
\Dorfman{e_1,e_2}_{\omega}=\Dorfman{e_1,e_2}+ \iota_{\rho(e_2)}\iota_{\rho(e_1)}\omega.
\end{equation}
Now, a long but straightforward computation shows that
\[
\{e_1,\{e_2,e_3\}_\omega\}_\omega-\{\{e_1,e_2\}_\omega,e_3\}_\omega-\{e_2,\{e_1,e_3\}_\omega\}_\omega = \iota_{\rho(e_3)}\iota_{\rho(e_2)}\iota_{\rho(e_1)} \jetd \omega
\]
For all $e_1, e_2, e_3 \in \Gamma (\omni)$.
In particular, $(\Gamma(\omni), \{-,-\}_\omega)$ is a Leibniz algebra if and only if $\jetd\omega=0$, i.e.~$\omega = \jetd \mu$ for some $\mu \in \Omega^{n+1}_{\jet E}$.
\end{remark}

\section{Higher Dirac-Jacobi structures}

In this section we study isotropic and involutive subbundles of a higher omni-Lie algebroid. So, let $E \to M$ be a vector bundle, let $n$ be a positive integer and let $\omni = \omni_n (E)$ be the higher omni-Lie algebroid of $E$.

\begin{Definition}
A subbundle $L \subset \omni$ is \emph{isotropic} if $(e_1, e_2)_+ = 0$ for all $e_1, e_2 \in L$ and it is \emph{involutive} when $\Dorfman{e_1, e_2} \in \Gamma (L)$ for all $e_1, e_2 \in \Gamma (L)$. A maximal isotropic and involutive subbundle of $\omni$ will be called a \emph{higher Dirac-Jacobi structure}.
\end{Definition}

\begin{remark}\label{rem:maximal}
Clearly, an isotropic subbundle $L \subset \omni$ is maximal isotropic if and only if, for every section $e \in \Gamma(\omni)$ such that $(e, e')_+ = 0$ for all $e' \in \Gamma (L)$, $e$ already belongs to $\Gamma (L)$.
\end{remark}

When $E = \ell$ is a line bundle, and $n = 1$, then higher Dirac-Jacobi structures $L \subset \omni$ are exactly Dirac-Jacobi structures, first studied in \cite{Wade2000, Wade2004} (see also \cite{Vitagliano}). Dirac-Jacobi structure encompass Jacobi structures, homogeneous Poisson structures, and hyperplane distributions as special cases \cite{Wade2000, Wade2004, Vitagliano}. For generic $n$ we will only consider higher Dirac-Jacobi structures projecting isomorphically on either $\dev E$ or $\jet_n E$.

\subsection{Higher Dirac-Jacobi structures projecting isomorphically on $\dev E$}

We begin with a form $\mu \in \Omega^{n+1}_{\jet E}$. In the following we denote by $B_\mu : \dev{E} \to \jet_n E$ the vector bundle map $\frkd \mapsto i_{\frkd} \mu$. The graph of $B_\mu$ is
\[
\operatorname{graph} B_\mu \defbe \left\{ \frkd + B_\mu(\frkd) : \frkd \in \dev E \right\} \subset \omni.
\]

\begin{theorem}
Let $\mu \in \Omega^{n+1}_{\jet E}$. The graph $L$ of $B_\mu$ is a maximal isotropic subbundle and every isotropic subbundle projecting isomorphically onto $\dev{E}$ arises in this way. Additionally, $L$ is involutive, hence a higher Dirac-Jacobi structure, if and only if $\mu$ is closed, hence exact.
\end{theorem}

\begin{proof}
Let $\mu \in \Omega^{n+1}_{\jet E}$. It is obvious that $\operatorname{graph} B_\mu \subset \omni$ is an isotropic subbundle. It is also easy to see, using Remark \ref{rem:maximal}, that $\operatorname{graph} B_\mu$ is maximal isotropic. Conversely, let $L \subset \omni$ be an isotropic subbundle projecting isomorphically onto $\dev E$. Then $L$ is the graph of a, necessarily unique, vector bundle map $B : \dev E \to \jet_{n} E$. From isotropicity $B$ is skewsymmetric in the sense that
\[
\iota_{\frkd_1} B(\frkd_2) + \iota_{\frkd_2} B(\frkd_1) = 0
\]
for all $\frkd_1, \frkd_2 \in \dev E$. Hence $B = B_\mu$ for some $\mu \in \Omega^{n+1} (\dev E, E)$ and it remains to check that $\mu \in \Omega^{n+1}_{\jet E}$. So let $\Phi \in \gl (E)$, let $\frkd_1, \ldots, \frkd_n \in \dev E$, and compute
\[
\begin{aligned}
\iota_\Phi \mu (\frkd_1, \frkd_2, \ldots, \frkd_n) & = \mu (\Phi, \frkd_1, \frkd_2, \ldots, \frkd_n) \\
& = - \mu (\frkd_1, \Phi, \frkd_2, \ldots, \frkd_n) \\
& = - B(\frkd_1) (\Phi, \frkd_2, \ldots, \frkd_n) \\
& = - \Phi \circ \lambda_{B(\frkd_1)} (\jd (\frkd_2), \ldots, \jd (\frkd_n))
\end{aligned}
\]
Next we check that $\lambda_{B(\frkd_1)} (\jd \frkd_2, \ldots, \jd \frkd_n)$ is skew-symmetric in its arguments $\frkd_1, \ldots, \frkd_n$. Indeed
\[
\lambda_{B(\frkd_1)} (\jd \frkd_2, \ldots, \jd \frkd_n) = B(\frkd_1)(\Id_E, \frkd_2, \ldots, \frkd_n) = - \iota_{\Id_E} \mu (\frkd_1, \ldots, \frkd_n)
\]
which is skewsymmetric. This shows that there exists an $E$-valued $(n+1)$-form $\nu \in \Omega^{n+1} (M, E)$ such that $\lambda_{B(\frkd_1)} (\jd (\frkd_2), \ldots, \jd (\frkd_n)) = \nu (\jd (\frkd_1),\jd (\frkd_2), \ldots, \jd (\frkd_n))$ for all $\frkd_1, \ldots, \frkd_n \in \dev E$. Hence $\mu \in \Omega^{n+1}_{\jet E}$ and $\nu = - \lambda_\mu$.

For the second part of the statement, let $\mu \in \Omega^{n+1}_{\frkJ E}$, and let $e_1, e_2 \in \Gamma (B_\mu)$. So there are $\frkd_1, \frkd_2$ such that $e_i = \frkd_i + \iota_{\frkd_i} \mu$, $i = 1,2$, and
\[
\{ e_1, e_2 \} = [\frkd_1, \frkd_2]_\dev + \LieDerivation_{\frkd_1} \iota_{\frkd_2} \mu - \LieDerivation_{\frkd_2} \iota_{\frkd_2} + \jetd \iota_{\frkd_2} \mu.
\]
The latter belongs to $\operatorname{graph} B_\omega$ if and only if
\begin{equation}\label{eq:Luca_1}
\LieDerivation_{\frkd_1} \iota_{\frkd_2} \mu - \LieDerivation_{\frkd_2} \iota_{\frkd_2} + \jetd \iota_{\frkd_2} \mu = \iota_{[\frkd_1, \frkd_2]_\dev} \mu.
\end{equation}
Using \eqref{Cartan} we see that \eqref{eq:Luca_1} is equivalent to $\iota_{\frkd_1} \iota_{\frkd_2} \jetd \mu = 0,$
and, from the arbitrariness of $\frkd_1, \frkd_2$, this is equivalent to $\jetd \mu = 0$.
\end{proof}

The above proposition shows that higher Dirac-Jacobi structures of $\omni$ projecting isomorphically onto $\dev E$ are in one-to-one correspondence with closed, hence exact, forms in $\Omega^{n+1}_{\jet E}$. In their turn, \eqref{eq:d} shows that exact elements in $\Omega^{n+1}_{\jet E}$ are in one-to-one correspondence with $E$-valued $n$-forms on $M$. Summarizing, the assignment
\[
\nu \longmapsto \operatorname{graph} B_{\jetd \jd^\ast \nu}
\]
is a one-to-one correspondence between $\Omega^n (M, E)$ and higher Dirac-Jacobi structures projecting isomorphically onto $\dev E$. The case when $E = \ell$ is a line bundle is particularly interesting. In this case, higher Dirac-Jacobi structures projecting isomorphically onto $\dev \ell$ encompass (\emph{pre}-)\emph{multicontact structures}, i.e.~corank $n$ distributions on $M$ \cite{Vitagliano2}, as we now show.

Begin with an $\ell$-valued $n$-form $\nu \in \Omega^n (M, \ell)$. By definition, the \emph{kernel} of $\nu$ is the (non-necessarily smooth) distribution $\ker \nu \subset TM$ on $M$ consisting of tangent vectors $v$ such that $\iota_v \nu = 0$.

\begin{Definition}\label{def:multicont}
An $\ell$-valued $n$-form $\nu$ on $M$ is of \emph{multicontact type} if $\ker \nu$ has corank exactly equal to $n$.
\end{Definition}

Notice that the kernel of an $\ell$-valued $n$-form $\nu$ of multicontact type is a smooth and regular distribution. When $n = 1$, Definition \ref{def:multicont} simply says that $D \defbe \ker \nu$ is an hyperplane distribution, hence $\nu$ is an everywhere non-zero, hence surjective, $\ell$-valued $1$-form on $M$, and can be interpreted as the projection $TM \to TM/D \cong \ell$. Next proposition generalizes this picture to the possibly higher $n$ case.

\begin{proposition}
Let $\nu$ be an $\ell$-valued $n$-form of multicontact type. Then $D \defbe \ker \nu$ is a smooth and regular distribution and there is a canonical isomorphism $\wedge^n (TM/D) \cong \ell$. Conversely, let $D \subset TM$ be a corank $n$ distribution, so that $\ell = \wedge^n (TM/D)$ is a line bundle. Then there exists a canonical $\ell$-valued $n$-form of multicontact type $\nu_D$ such that $D = \ker \nu_D$. Correspondences
\[
\nu \mapsto \ker \nu, \quad \text{and} \quad D \mapsto \nu_D
\]
are mutually inverse.
\end{proposition}

\begin{proof}
Let $\ell \to M$ be a line bundle and let $\nu \in \Omega^n (M, \ell)$ be of multicontact type. Denote $D \defbe \ker \nu$, and consider the vector bundle map
\[
A : \wedge^n (TM/D) \to \ell, \quad (v_1 + D) \wedge \cdots \wedge (v_n + D) \mapsto \nu (v_1, \ldots, v_n).
\]
Clearly $A$ is an isomorphism. Conversely, let $D \subset TM$ be a corank $n$ distribution. Put $\ell = \wedge^n (TM/D)$ and let $\nu_D : \wedge^n (TM/D) \to \ell$ be the $n$-form defined by
\[
\nu_D (v_1, \ldots, v_n) \defbe (v_1 +D) \wedge \cdots \wedge (v_n + D).
\]
It is easy to see that $\ker \nu_D = D$, in particular $\nu_D$ is of multicontact type. This concludes the proof.
\end{proof}

\subsection{Higher Dirac-Jacobi structures projecting isomorphically on $\jet_{n} E$}

We now pass to higher Dirac-Jacobi structures $L$ projecting isomorphically on $\jet_{n} E$. In particular, $L$ is the graph of a vector bundle map $B : \jet_n E \to \dev E$:
\[
L = \operatorname{graph} B \defbe \{ B(\mu) + \mu : \mu \in \jet_n E \} \subset \omni.
\]
We want to characterize isotropicity and involutivity of $L$ in terms of $B$.
The case $n = 1$ is studied in details in \cite{Chen-Liu-Sheng2} (see also \cite{Vitagliano} for the case $n = \operatorname{rank} E = 1$): in this case $B$ can be seen as a first order bidifferential operator $\Delta : \Gamma (E) \times \Gamma (E) \to \Gamma (E)$ via
\[
\Delta (e_1, e_2) = B(\jetd e_1)(e_2).
\]
In particular, $\Delta$ is a derivation in the first entry. Then $L$ is isotropic if and only if $\Delta$ is skew-symmetric, hence it is a bi-derivation. Additionally, $L$ is a Dirac subbundle if and only if 1) $\Delta$ is a Jacobi bracket on $\ell$, if $E = \ell$ is a line bundle, and 2) $\Delta$ is the Lie bracket of a, necessarily unique, Lie algebroid structure on $E$, if $\operatorname{rank} E > 1$. See \cite{Chen-Liu-Sheng2} for all the details.

Next we assume $n > 1$. We have the following

\begin{theorem}\label{theor:n<dimM}
Let $L \subset \omni$ be an isotropic subbundle projecting isomorphically onto $\jet_n E$. If $1 < n < \dim M + 1$, then $L = 0 \oplus \jet_n E$, and, in this case, $L$ is necessarily a higher Dirac-Jacobi structure.
\end{theorem}

\begin{proof}
See Appendix \ref{appendix:proof}.
\end{proof}

It remains to study the case $n = \dim M + 1$. In order to state the main result of this section, we need to give a new definition. So, let $E \to M$ be a vector bundle.

\begin{Definition}
A \emph{volume with values Lie algebras} is a section $Z$ of
\[
\wedge^{\mathrm{top}}TM \otimes \Hom(\wedge^2 E, E)
\]
such that, for every top form $\Omega \in \Omega^{\mathrm{top}} (M)$,
$
\conpairing{\Omega, Z} \in \Gamma (\Hom(\wedge^2 E, E))
$
 gives to $E$ the structure of a bundle of Lie algebras.
\end{Definition}

In the next proposition we show that, when $n = \dim M +1$, a higher Dirac-Jacobi structure in $\omni$ projecting isomorphically onto $\jet_n E$ is equivalent to a volume with values Lie algebras $Z$. Before stating our result we need some remarks. Let $m = \dim M$, and notice that, in view of \eqref{eq:iso},
\[
\jet_{m+1} E \cong \wedge^m T^\ast M \otimes E.
\]
In this case, a vector bundle map $B : \jet_n E \to \dev E$ can be seen as a vector bundle map $B : \wedge^m T^\ast M \otimes E \to \dev E$ or, equivalently, as a section $Z$ of $\wedge^{\mathrm{top}}TM \otimes \Hom(\wedge^2 E, E)$.
Now the bundle $\Hom(\wedge^2 E, E)$ embeds canonically into $\Hom (E, \dev E)$ via
\[
\Hom(\wedge^2 E, E) \hookrightarrow \Hom (E, \dev E), \quad b \mapsto \phi_b
\]
where
\[
\phi_b (e_1) (e_2) \triangleq b(e_1, e_2), \quad e_1, e_2 \in E.
\]
\begin{theorem}
Let $n = m +1 = \dim M + 1$, and assume that $Z$ be a section of  $\wedge^{m}TM \otimes \Hom(\wedge^2 E, E)$. Denote by $B_Z : \jet_{m+1} E \to \dev E$ the associated vector bundle map. The graph $L$ of $B_Z$ is a maximal isotropic subbundle of $\omni$ and every isotropic subbundle projecting isomorphically onto $\jet_{m+1} E$ arises in this way. Additionally, $L$ is involutive, hence a higher Dirac-Jacobi structure, if and only if $Z$ is a volume with values Lie algebras.
\end{theorem}

\begin{proof}
As in Remark \ref{rem:split} we denote by
\[
\frkj_{\bullet} : \Omega^{\bullet -1}(M,E) \to \Omega^\bullet_{\jet E}, \quad \lambda \mapsto \jetd \jd^\ast \lambda
\]
the embedding. As already remarked, $\frkj_{m+1}$ is actually a $C^\infty (M)$-linear isomorphism. Hence it comes from a vector bundle isomorphism that we denote again by
\[
\frkj_{m+1} : \wedge^m T^\ast M \otimes E \to \jet_{m+1} E.
\]
Now, let $L \subset \omni$ be a subbundle projecting isomorphically onto $\jet_{m+1} E$. So $L$ is the graph of a vector bundle map $B : \jet_{m+1} E \to \dev E$. We work locally and fix a volume form $\operatorname{vol}$ on $M$. Clearly $B$ is completely determined by the composition
\[
\xymatrix{E \ar[r]^-{\operatorname{vol} \otimes -} &  \wedge^m T^\ast M \otimes E \ar[r]^-{\frkj_{m+1}} & \jet_{m+1} E \ar[r]^-B & \dev E},
\]
that we denote by
\[
B_E : E \to \dev E.
\]
Notice that
\begin{equation}\label{eq:B_E}
\frkj_{m+1} (\operatorname{vol} \otimes e) = (-1)^{m} \operatorname{vol} \wedge \jetd e
\end{equation}
for all $e \in \Gamma (E)$.
First we want to show that $L$ is isotropic if and only if $B = B_Z$ for some section $Z$ of $\wedge^m TM \otimes \Hom(\wedge^2 E, E)$ (and, in this case, it is maximal isotropic). This is equivalent to the following two conditions:
\[
\sigma \circ B_E = 0, \quad \text{and} \quad B_E (e_1)(e_2) + B_E (e_2)(e_1)=0
\]
for all $e_1, e_2 \in \Gamma(E)$, and, in this case, $\conpairing{\operatorname{vol}, Z} (e_1,e_2) = B_E (e_1)(e_2)$. Now, $L$ is isotropic if and only if
\[
\conpairing{B(\mu_1), \mu_2}_{\jet E} +  \conpairing{B(\mu_2), \mu_1}_{\jet E} = 0
\]
for all $\mu_1, \mu_2 \in \jet_{m+1} E$. In view of the above discussion, it is enough to consider $\mu_1, \mu_2$ of the form
\[
\mu_i = \frkj_{m+1} (\operatorname{vol} \otimes e_i) = (-1)^m \operatorname{vol}\wedge \jetd e_i , \quad e_i \in \Gamma (E), \quad i = 1,2.
\]
where we used \eqref{eq:B_E}. In this case,
\[
\begin{aligned}
& \conpairing{B(\mu_1), \mu_2}_{\jet E} +  \conpairing{B(\mu_2), \mu_1}_{\jet E} \\
& = (-1)^m \left(\iota_{B_E (e_1)} \left(\operatorname{vol}\wedge \jetd e_2\right)+  \iota_{B_E (e_2)} \left(\operatorname{vol}\wedge \jetd e_1\right) \right)\\
& = (-1)^m \left( \left(\iota_{\sigma (B_E(e_1))} \operatorname{vol}\right)\wedge \jetd e_2+  \left(\iota_{\sigma (B_E(e_2))}\operatorname{vol}\right)\wedge \jetd e_1\right) \\
& \quad +\operatorname{vol} \wedge \left(B_E(e_1)(e_2) + B_E(e_2)(e_1) \right).
\end{aligned}
\]
Denote by $\mu \in \Omega^m_{\jet E}$ the latter expression. It vanishes if and only if
\[
\iota_{\Id_E} \mu = 0, \quad \text{and} \quad \iota_{\Id_E} \jetd \mu = 0.
\]
The first condition reads
\begin{equation}\label{eq:sigma_B}
\left(\iota_{\sigma (B_E(e_1))}\operatorname{vol}\right) \otimes e_2 + \left(\iota_{\sigma(B_E(e_2))}\operatorname{vol}\right) \otimes e_1 = 0,
\end{equation}
while the second condition simply reads
\[
\operatorname{vol}\wedge \left(B_E(e_1)(e_2) + B_E(e_2)(e_1) \right).
\]
or, which is the same, simply
\[
B_E(e_1)(e_2) + B_E(e_2)(e_1) = 0.
\]
So, it remains to show that condition \eqref{eq:sigma_B} holds for all $e_1,e_2$ if and only if $\sigma \circ B_E = 0$. The ``if'' part is obvious. For the ``only if'' part choose a local basis $(\partial_i)$ for $\mathfrak X (M)$, and a local basis $(\varepsilon_\alpha)$ for $\Gamma (E)$. Then there are local functions $B_\alpha^i$ such that
\[
\sigma(B_E(\varepsilon_\alpha)) = B^i_\alpha \partial_i.
\]
Finally, write condition \eqref{eq:sigma_B} for $e_1 = \varepsilon_\alpha$ and $e_2 = \varepsilon_\beta$ to get
\[
\iota_{\partial_i} \operatorname{vol} \otimes \left( B_\alpha^i \varepsilon_\beta + B_\beta^i \varepsilon_\alpha \right) = 0 \ \Longrightarrow \  B_\alpha^i \varepsilon_\beta + B_\beta^i \varepsilon_\alpha = 0,
\]
which in turn implies $B_\alpha^i = 0$ for all $\alpha, i$, hence $\sigma \circ B_E = 0$. An easy check reveals that $L$ is also a maximal isotropic subbundle. This concludes the proof of the first part of the statement.

For the second part, let $Z$ be a section of $\wedge^m TM \otimes \Hom (\wedge^2 E, E)$, let $B_Z : \jet_{m+1} E \to \dev E$ be the associated vector bundle map, and let $L \subset \omni$ be its graph. Then $L$ is involutive if and only if, for all $\mu_1, \mu_2 \in \Omega^{m+1}_{\jet E}$,
\[
\Dorfman{B_Z(\mu_1) + \mu_1, B_Z(\mu_2) + \mu_2} \in \Gamma (L).
\]
Now, as $\mu_2$ has top degree
\[
\Dorfman{B_Z(\mu_1) + \mu_1, B_Z(\mu_2) + \mu_2} = [B_Z(\mu_1), B_Z(\mu_2)] + \LieDerivation_{B_Z(\mu_1)} \mu_2,
\]
which is in $\Gamma (L)$ if and only if
\begin{equation}\label{eq:B_Z}
B_Z(\LieDerivation_{B_Z(\mu_1)} \mu_2) = [B_Z(\mu_1), B_Z(\mu_2)].
\end{equation}
We fix again a local volume form $\operatorname{vol}$ so that $\mu_i = \operatorname{vol} \wedge \jetd e_i$, $e_i = \Gamma (E)$, $i = 1,2$. In particular
\begin{equation} \label{eq:B_Z_2}
\LieDerivation_{B_Z(\mu_1)} \mu_2 = \LieDerivation_{B_Z(\mu_1)} \operatorname{vol} \wedge \jetd e_2 = \operatorname{vol} \wedge \jetd B_Z(\mu_1)(e_2),
\end{equation}
where we used that $\sigma \circ B_Z = 0$. Finally, we put
\[
b \defbe \conpairing{\operatorname{vol}, Z} : \wedge^2 E \to E,
\]
hence $B_Z(\mu_i) = b(e_i, -)$, and, taking into account \eqref{eq:B_Z_2}, condition \eqref{eq:B_Z} reads
\[
[b(e_1, -), b(e_2, -)] = b(b(e_1,e_2), -).
\]
Applying both sides to a third section $e_3 \in \Gamma (E)$, after reorganizing the three terms, we get
\[
b(e_1, b(e_2, e_3)) + b(e_2, b(e_3, e_1)) + b(e_3, b(e_1,e_2)) = 0,
\]
which is the Jacobi identity for $b$. This concludes the proof.
\end{proof}


\section{Higher extended generalized tangent bundle}

When $E = \Real_M := M \times \Real$ is the trivial line bundle, the structure maps in the higher omni-Lie algebroid can be expressed entirely in terms of standard Cartan calculus. This expressions are sometimes useful, and we present them in this section. This section will also serve as a dictionary from the trivial \cite{Iglesias2002, Iglesias2006, Leon1, Leon1997, Wade2000} to the non-trivial \cite{Kirillov1976, Vitagliano} line bundle case.

We begin with a more general situation which is of an independent interest. Namely, we assume $E \to M$ is a generic vector bundle equipped with a linear connection $\nabla$. In this case, the short exact sequence \eqref{eq:SES_dev} splits canonically via
\[
\nabla: TM \to \dev E, \quad X \mapsto \nabla_X.
\]
Accordingly, there is a direct sum decomposition of vector bundles
\begin{equation}\label{eq:nabla_split}
\dev E \cong TM \oplus \gl (E)
\end{equation}
and every derivation $\frkd \in \Gamma (\dev E)$ can be uniquely written in the form
\begin{equation}\label{eq:partial_split}
\frkd = \nabla_X + \Phi
\end{equation}
where $X = \jd (\frkd)$ and $\Phi$ is a section of $\gl (E)$. In what follows we will often use \eqref{eq:nabla_split} to identify $\frkd$ with the pair $(X, \Phi)$.

Now, recall that the short exact sequence \eqref{eq:SES_Omega} splits in the category of graded vector spaces, but does not split canonically in the category of graded $\Omega^\bullet (M)$-modules. However, connection $\nabla$ defines an $\Omega^\bullet (M)$-linear splitting:
\[
\frkj_\bullet^\nabla : \Omega^{\bullet -1} (M, E) \to \Omega^\bullet_{\jet E}, \quad \lambda \mapsto \frkj_\bullet \lambda - \e_\bullet (d^\nabla \lambda),
\]
where $d^\nabla : \Omega^{\bullet-1} (M, E) \to \Omega^\bullet (M, E)$ is the connection ``differential''. As a consequence, there is an isomorphism of graded $\Omega^\bullet (M)$-modules
\begin{equation}\label{eq:iso_nabla}
\Omega^\bullet_{\jet E} \cong \Omega^\bullet (M, E) \oplus \Omega^{\bullet-1}(M,E),
\end{equation}
identifying $\mu \in \Omega^k_{\jet E}$ with a pair consisting of an $E$-valued $k$-form, and an $E$-valued $(k-1)$-form, that we denote by $(\widetilde{\mu}_0, \widetilde{\mu}_1)$ in order to distinguish it from the pair $(\mu_0, \mu_1)$ of Section \ref{Sec:vect_form}. Using the same notation as in Section \ref{Sec:vect_form}, it is easy to see that
\begin{equation}\label{eq:mu_nabla}
(\widetilde{\mu}_0, \widetilde{\mu}_1) = (\mu_0 + d^\nabla \mu_1, \mu_1).
\end{equation}
Now, we want to describe all natural operations on $\Omega^\bullet_{\jet E}$ in terms of the isomorphisms \eqref{eq:nabla_split} and \eqref{eq:iso_nabla}. To do this, we first recall two basic facts about linear connections and their connection differential.

Firstly, any $\gl(E)$-valued form on $M$, $\Phi \in \Omega^\bullet(M, E)$, defines a degree $|\Phi|$, graded homomorphism of $\Omega^\bullet (M)$-modules, also denoted by
\[
\Phi : \Omega^\bullet (M, E) \to \Omega^{\bullet + |\Phi|} (M, E),
\]
in the obvious way. In particular, the curvature $R$ of $\nabla$ is a $\gl (E)$-valued $2$-form on $M$, hence it defines a degree $2$, graded homomorphism
\[
R : \Omega^\bullet (M, E) \to \Omega^{\bullet + 2} (M, E),
\]
and we have $(d^\nabla)^2 = R$.

Secondly, the connection $\nabla$ defines a connection in the whole tensor algebra of $E$, in particular on $\gl (E)$. We denote again by $\nabla$ the induced connection. For any $\Phi \in \Omega^\bullet (M, \gl (E))$ we have
\begin{equation}\label{eq:d_nablaPhi}
d^\nabla \Phi = [d^\nabla, \Phi],
\end{equation}
where, in the left hand side, $d^\nabla$ is the differential of the connection in $\gl (E)$, $\Phi$ is a $\gl(E)$-valued form, but $d^\nabla \Phi$ is interpreted as a graded homomorphism $\Omega^\bullet (M, E) \to \Omega^{\bullet + |\Phi| + 1} (M, E)$, while, in the right hand side, $d^\nabla$ is the differential of the connection in $E$ and $\Phi$ is interpreted as a graded homomorphism $\Omega^\bullet (M, E) \to \Omega^{\bullet + |\Phi|} (M, E)$.

Finally, let $\omega \in \Omega^\bullet (M)$, and let $\frkd = \nabla_X + \Phi \in \Gamma (\dev E)$. The multiplication by $\omega$, the de Rham differential, the contraction and the Lie derivative along $\frkd$ induce new operations on $\Omega^\bullet (M, E) \oplus \Omega^{\bullet-1} (M, E)$ via \eqref{eq:iso_nabla}. The latter will be denoted by
\[
\omega \, \widetilde{\wedge}\,  -, \ \widetilde{\jetd},\  \widetilde{\iota}_{(X, \Phi)},\ \widetilde{\LieDerivation}_{(X, \Phi)},
\]
 respectively, in order to distinguish them from the operations on $\Omega^\bullet (M, E) \oplus \Omega^{\bullet-1} (M, E)$ discussed in Section (\ref{Sec:vect_form}) (and independent of $\nabla$). A direct computation exploiting \eqref{eq:partial_split}, \eqref{eq:mu_nabla} and \eqref{eq:d_nablaPhi} shows that, for all $\mu \in \Omega^\bullet_{\jet E}$,
\begin{align}
\omega \, \widetilde{\wedge} \, (\widetilde{\mu}_0, \widetilde{\mu}_1) & = (\omega \wedge \widetilde{\mu}_0, (-1)^{|\omega|} \omega \wedge \widetilde{\mu}_1) \label{eq:tilde1}\\
\widetilde{\jetd}(\widetilde{\mu}_0, \widetilde{\mu}_1) & = (d^\nabla \widetilde{\mu}_0 - R(\widetilde{\mu}_1), \widetilde{\mu}_0 - d^\nabla \widetilde{\mu}_1) \label{eq:tilde2}\\
 \widetilde{\iota}_{(X, \Phi)}(\widetilde{\mu}_0, \widetilde{\mu}_1) & = (\iota_X \widetilde{\mu}_0 + \Phi (\widetilde{\mu}_1), - \iota_X \widetilde{\mu}_1) \label{eq:tilde3}\\
\widetilde{\LieDerivation}_{(X, \Phi)}(\widetilde{\mu}_0, \widetilde{\mu}_1) & = \left(\LieDerivation^\nabla_X \widetilde{\mu}_0 + \Phi (\widetilde{\mu}_0) + \left(d^\nabla \Phi - \tfrac{1}{2}\iota_X R \right)(\widetilde{\mu}_1), \LieDerivation^\nabla_X \widetilde{\mu}_1 + \Phi (\widetilde{\mu}_1)\right)\label{eq:tilde4}
\end{align}
where $\LieDerivation^\nabla_X = [\iota_X, d^\nabla]$. In particular
\[
\widetilde{\iota}_{(0, \Id_E)} (\widetilde{\mu}_0, \widetilde{\mu}_1) = (\widetilde{\mu}_1, 0).
\]

When $E = \Real_M$ is the trivial line bundle and $\nabla$ is the trivial connection on it, Equations \eqref{eq:tilde2}--\eqref{eq:tilde4} significantly simplify. Indeed, in this case $\gl (E) = \Real_M$ and every derivation $\frkd \in \Gamma (\dev \Real_M)$ is of the form $\frkd = X + f$ where $X = \jd (\frkd)$ as usual, and $f \in C^\infty (M)$. Additionally, $\widetilde{\mu}_0, \widetilde{\mu}_1$ are standard differential forms, $d^\nabla = d$, the standard de Rham differential, and $R = 0$. Hence, in this case,
\begin{align}
\widetilde{\jetd}(\widetilde{\mu}_0, \widetilde{\mu}_1) & = (d \widetilde{\mu}_0, \widetilde{\mu}_0 - d \widetilde{\mu}_1) \label{eq:tilde5}\\
 \widetilde{\iota}_{(X, f)}(\widetilde{\mu}_0, \widetilde{\mu}_1) & = (\iota_X \widetilde{\mu}_0 + f\widetilde{\mu}_1, - \iota_X \widetilde{\mu}_1) \label{eq:tilde6}\\
\widetilde{\LieDerivation}_{(X, f)}(\widetilde{\mu}_0, \widetilde{\mu}_1) & = \left(\LieDerivation_X \widetilde{\mu}_0 + f \widetilde{\mu}_0 + df \wedge \widetilde{\mu}_1, \LieDerivation_X \widetilde{\mu}_1 + f \widetilde{\mu}_1\right) \label{eq:tilde7}
\end{align}

\begin{remark}
The operations \eqref{eq:tilde5}, \eqref{eq:tilde6} and \eqref{eq:tilde7} are also discussed in \cite{Wade2000, Iglesias2001, Iglesias2004}. Our version differs from those by a sign, which is due to our conventions about the isomorphism $\jet_n \Real_M \cong \wedge^n T^\ast M \oplus \wedge^{n-1} T^\ast M$.
\end{remark}

\begin{Definition}
The higher extended generalized tangent bundle is the higher omni-Lie algebroid
\begin{equation}\label{eq:higher_gen}
\omni \cong \left( TM \oplus \Real_M \right) \otimes (\wedge^n T^\ast M \oplus \wedge^{n-1} T^\ast M).
\end{equation}
of the trivial line bundle $\Real_M$.
\end{Definition}

\begin{remark}
In \eqref{eq:higher_gen} we used the vector bundle splittings $\dev \Real_M \cong TM \oplus \Real_M$ and $\jet_n \Real_M \cong \wedge^n T^\ast M \oplus \wedge^{n-1} T^\ast M$, induced by the trivial connection in the trivial line bundle. For $n = 1$, we recover Wade's \emph{extended generalized tangent bundle} \cite{Wade2000}.
\end{remark}

We conclude this section describing explicitly the structure maps of the higher generalized tangent bundle in terms of standard Cartan calculus, as promised. Using \eqref{eq:tilde5}, \eqref{eq:tilde6} and \eqref{eq:tilde7}, we immediately see that, for any pair of sections $(X, f) + (\widetilde \mu_0, \widetilde \mu_1), (Y, g) + (\widetilde \nu_0, \widetilde \nu_1)$ of the higher generalized tangent bundle, we have
\[
\begin{aligned}
& \left((X, f) + (\widetilde \mu_0, \widetilde \mu_1), (Y, g) + (\widetilde \nu_0, \widetilde \nu_1)\right)_+ \\
&= \frac{1}{2}\left(\iota_X \widetilde{\nu}_0 + f\widetilde{\nu}_1 + \iota_Y \widetilde{\mu}_0 + g\widetilde{\mu}_1, - \iota_X \widetilde{\nu}_1 - \iota_Y \widetilde{\mu}_1\right)
\end{aligned}
\]
and
\[
\begin{aligned}
& \Dorfman{(X, f) + (\widetilde \mu_0, \widetilde \mu_1), (Y, g) + (\widetilde \nu_0, \widetilde \nu_1)} = \left([X,Y], X(g) - Y(f)\right) \\
& +
\left(\LieDerivation_X \widetilde{\nu}_0 + f \widetilde{\nu}_0 + df \wedge \widetilde{\nu}_1 - \iota_Y d \widetilde \mu_0 - g (\widetilde \mu_0 - d \widetilde \mu_1), \LieDerivation_X \widetilde{\nu}_1 + f \widetilde{\nu}_1 + \iota_Y (\widetilde \mu_0 - d \widetilde \mu_1)\right)
\end{aligned}
\]
These formulas generalize to the case $n \geq 1$ those of \cite{Wade2000} (in the case $n = 1$), up to a conventional sign.

\begin{remark}
Recall that the graph of a map $B = \dev E \to \jet_n E$ is a higher Dirac-Jacobi structure if and only if $B$ is the flat map
\[
\mu_\flat : \dev E \to \jet_n E, \quad \frkd \mapsto \iota_\frkd \mu
\]
of a closed, hence exact form $\mu \in \Omega^{n+1}_{\jet E}$. We now apply this result to the case $E = \Real_M$ to see, from \eqref{eq:tilde5} and \eqref{eq:tilde6}, that the graph of a map $B : TM \oplus \Real_M \to \wedge^n T^\ast M \oplus \wedge^{n-1} T^\ast M$ is a higher Dirac-Jacobi structure if and only if $B$ is the flat map
\[
(\widetilde \mu_0, \widetilde \mu_1)_\flat : TM \oplus \Real_M \to \wedge^n T^\ast M \oplus \wedge^{n-1} T^\ast M, \quad (X, f) \mapsto \widetilde \iota_{(X,f)} (\widetilde \mu_0, \widetilde \mu_1)
\]
of a pair $(\widetilde \mu_0, \widetilde \mu_1) \in \Omega^{n+1} (M) \oplus \Omega^n (M)$ such that $\widetilde \jetd (\widetilde \mu_0, \widetilde \mu_1) = 0$, or, which is the same,
$
\widetilde \mu_0 = d \widetilde \mu_1.
$
\end{remark}

\color{black}
\appendix

\section{Proof of Theorem \ref{theor:n<dimM}}\label{appendix:proof}

In this proof we understand the embedding
\[
\e_\bullet : \Omega^\bullet (M, E) \to \Omega^\bullet_{\jet E}
\]
and interpret $\Omega^\bullet (M, E)$ as a graded $\Omega^\bullet (M)$-submodule in $\Omega^\bullet_{\jet E}$. We work in local coordinates. So, let $\dim M = m$, and let $x^1, \ldots, x^m$ be coordinates on $M$, and $\partial_i := \frac{\partial}{\partial x^i}$ the associated coordinate vector fields. Additionally, let $(\varepsilon_\alpha)$ be a local frame of sections of $E$. Locally, a derivation $\frkd \in \Gamma (\dev E)$ can be uniquely written as
$
\frkd = X^i \partial_i + X^\gamma_\beta \varepsilon^{\beta}_\gamma,
$
where,
\begin{enumerate}
\item[(1)]$\varepsilon^{\beta}_\gamma$ is the unique endomorphism $\Gamma (E) \to \Gamma (E)$ such that $\varepsilon^{\beta}_\gamma (\varepsilon_\alpha) = \delta^\beta_\alpha \varepsilon_\gamma$, for all $\alpha$,
\item[(2)] abusing the notation, we also denote by $\partial_i$ the unique derivation with symbol the $i$-th coordinate vector field, and such that $\partial_i \varepsilon_\alpha = 0$ for all $\alpha$,
\end{enumerate}
and the $X^i, X^\alpha_\beta$ are local functions. It will be also useful to consider the coordinate volume form
\[
\operatorname{vol} := dx^1 \wedge \cdots \wedge dx^m,
\]
and, for every (skewsymmetric) multiindex $I = i_1 \cdots i_k$, the $m-k$ form
\[
\operatorname{vol}_{I} = \operatorname{vol}_{i_1 \cdots i_k} := \iota_{\partial_{i_1}} \cdots \iota_{\partial_{i_k}} \operatorname{vol}.
\]
In the following we denote by $|I| := k$ the \emph{lenght} of a multiindex $I = i_1 \cdots i_k$.

It easily follows from \eqref{eq:iso} that $\Omega^\bullet_{\jet E}$ is locally generated, as a graded $\Omega^\bullet (M)$-module, by $(\varepsilon_\alpha, \jetd \varepsilon_\alpha)$. Hence, it is generated, as a $C^\infty (M)$-module, by the following forms:
\[
\mu_{I, \alpha} := \operatorname{vol}_{I} \otimes \varepsilon_\alpha, \quad \text{and} \quad \nu_{J, \alpha} := \operatorname{vol}_{J}\wedge \jetd \varepsilon_\alpha.
\]

Now, let $1 < n < m + 1$, and let $B : \jet_n E \to \dev E$ be a vector bundle map. Then $B$ is completely determined by
\[
B(\mu_{I, \alpha}) = B^i_{I, \alpha} \partial_i + B^\gamma_{I, \alpha\beta} \varepsilon^\beta_\gamma, \quad |I| = m-n,
\]
and
\[
B(\nu_{J, \alpha}) = C^i_{J, \alpha} \partial_i + C^\gamma_{J, \alpha\beta} \varepsilon^\beta_\gamma, \quad |J| = m - n + 1.
\]
Compute
\[
\begin{aligned}
\conpairing{B(\mu_{I, \alpha}), \mu_{I', \alpha'}}_{\jet E} & = B^i_{I, \alpha} \mu_{iI', \alpha'} \\
\conpairing{B(\mu_{I, \alpha}), \nu_{J, \alpha'}}_{\jet E} & = B^i_{I, \alpha} \nu_{iJ, \alpha'} + (-1)^{n-1} B^\gamma_{I, \alpha\alpha'} \mu_{J, \gamma} \\
\conpairing{B(\nu_{J, \alpha}), \mu_{I, \alpha'}}_{\jet E}& = C^i_{J, \alpha} \mu_{iI, \alpha'}  \\
\conpairing{B(\nu_{J, \alpha}), \nu_{J', \alpha'}}_{\jet E} & = C^i_{J, \alpha} \nu_{iJ', \alpha'} + (-1)^{n-1} C^\gamma_{J, \alpha\alpha'} \mu_{J', \gamma}.
\end{aligned}
\]

Denote by $L$ the graph of $B$ and recall that $L$ is an isotropic subbundle of $\omni$ if and only if
\begin{equation}\label{eq:Bmunu}
\conpairing{B(\mu), \nu}_{\jet E} + \conpairing{B(\nu), \mu}_{\jet E} = 0,\quad \text{for all $\mu, \nu \in \Omega^{n}_{\jet E}$}.
\end{equation}
It is clear that, if $B = 0$, then $L = 0 \oplus \jet_n E$ is isotropic, and it is easy to check that it is actually maximal isotropic and involutive, i.e.~a higher Dirac-Jacobi structure. Conversely, let $L$ be isotropic. Then, using \eqref{eq:Bmunu} with $\mu, \nu$ chosen among the generators $\mu_{I, \alpha}, \nu_{J, \alpha}$, we find, e.g.:
\[
0 = \conpairing{B(\mu_{I, \alpha}), \mu_{I', \alpha'}}_{\jet E} + \conpairing{B(\mu_{I', \alpha'}), \mu_{I, \alpha}}_{\jet E} = \left(\delta_{I'}^{I''} \delta_{\alpha'}^{\alpha''}B^i_{I, \alpha} + \delta_{I}^{I''} \delta_{\alpha}^{\alpha''}B^i_{I', \alpha'} \right)\mu_{iI'', \alpha''}.
\]
As the lenght of $I$ is at least $1$, it follows that $B^i_{I, \alpha} = 0$. One can show that the other coefficients $B^\gamma_{I, \alpha \beta}, C^i_{J, \alpha}, C^\gamma_{J, \alpha\beta}$ all vanish, in a similar way. We leave the obvious details to the reader.

\section*{Acknowledgements}
Part of this work was completed while YB and TZ were visiting the Chern Institute of Mathematics in September 2015. YB and TZ are grateful to Prof. Chengming Bai for his kind invitation and hospitality. We would like to thank Prof. Zhangju Liu for giving us the key idea at the very beginning of the work. We thank Prof. Zhuo Chen and  Prof. Yunhe Sheng for explaining to us the concept of omni-Lie algebroids. YB also thanks Department of Mathematics at Penn State University for its hospitality and China Scholarship Council (No.201608360038) for financial support. LV thanks YB and TZ for inviting him to join this project.



\begin{thebibliography}{99}




\bibitem{Bi2011}
  \newblock Bi, Y., and Y. Sheng,
  \newblock {\em On higher analogues of Courant algebroids},
\newblock {Sci. China Math.} \textbf{54} (2011), 437--447.

\bibitem{Y.Bi2015}
  \newblock Bi, Y., and Y. Sheng,
  \newblock {\em Dirac structures for higher analogues of Courant algebroids},
  \newblock {Int. J. Geom. Methods Mod. Phys.} \textbf{12} (2015), 1--13.

\bibitem{BAR2016}
  \newblock Bursztyn, H.,  N. Martinez Alba, and R. Rubio,
  \newblock {\em On higher Dirac structures},
  \newblock to appear in IMRN.
  \newblock e-print: arXiv:1611.02292.

\bibitem{Chen-Liu}
\newblock Chen, Z.,  and Z. -J. Liu,
 \newblock {\em Omni-Lie algebroids},
  \newblock {J. Geom. Phys.}  \textbf{60} (2010), 799--808.
   \newblock e-print: arXiv:0710.1923.

\bibitem{Chen-Liu-Sheng1}
 \newblock Chen, Z.,  Z. -J. Liu, and Y. Sheng,
 \newblock {\em $E$-Courant algebroids},
 \newblock {Int. Math. Res. Not.} \textbf{2010} (2010), 4334--4376.
      \newblock e-print: arXiv:0805.14093.

\bibitem{Chen-Liu-Sheng2}
 \newblock Chen, Z., Z. -J. Liu, and Y. Sheng,
 \newblock {\em Dirac structures  of omni-Lie algebroids},
 \newblock {Int. J. Math.} \textbf{22} (2011), 1163--1185.
      \newblock e-print: arXiv:0802.3919.



\bibitem{GSM2010}
\newblock Gracia-Saz, A., and R. A. Mehta,
\newblock {\em Lie algebroid structures on double vector bundles and representation theory of Lie algebroids},
\newblock {Adv. Math.} \textbf{223}(2010), 1236--1275.
\newblock e-print: arXiv:0810.0066.



\bibitem{Iglesias2001}
\newblock Iglesias-Ponte, D., and J. C. Marrero,
\newblock {\em Generalized Lie bialgebroids and Jacobi structures},
\newblock {J. Geom. Phys.} \textbf{40} (2001), 176--199.
\newblock e-print: arXiv:math/0008105.

\bibitem{Iglesias2002}
\newblock Iglesias-Ponte, D., and  J. C. Marrero,
\newblock  {\em Lie bialgebroid foliations and $\omni^1(M)$-Dirac structures},
\newblock  {J. Phys. A: Math. Gen.} \textbf{35} (2002), 4085--4104.
 \newblock e-print: arXiv:math/0106086.

 \bibitem{Iglesias2004}
\newblock Iglesias-Ponte, D., and  A. Wade,
\newblock {\em Contact manifolds and generalized complex structures},
\newblock {J. Geom. Phys.} \textbf{53} (2006), 249--258.
 \newblock e-print: arXiv:math/0404519.

\bibitem{Iglesias2006}
\newblock Iglesias-Ponte, D., and  A. Wade,
\newblock {\em Integration of Dirac-Jacobi structures},
\newblock {J. Phys. A: Math. Gen.} \textbf{39} (2006), 4181--4190.
 \newblock e-print: arXiv:math/0507538.

\bibitem{Kirillov1976}
\newblock Kirillov, A.,
\newblock {\em Local Lie algebras},
\newblock {Russian Math. Surveys} \textbf{31} (1976), 55--76.

\bibitem{Leon1}
\newblock de Le\'on, M.,  J. C. Marrero, and E. Padr\'on,
\newblock {\em Lichnerowicz-Jacobi cohomology},
\newblock {J. Phys. A: Math. Gen.} \textbf{30} (1997), 6029--6055.

\bibitem{Leon1997}
\newblock de Le\'on, M., J. C. Marrero, and E. Padr\'on,
\newblock {\em Lichnerowicz-Jacobi cohomology of Jacobi manifolds},
\newblock {C. R. Acad. Sci. Paris, Ser. I} \textbf{324} (1997), 71--76.


\bibitem{Sheng2017}
\newblock Liu, J.,  Y. Sheng, and C. Wang,
\newblock {\em Omni $n$-Lie algebras and linearization of higher analogues of Courant algebroids},
\newblock  {Int. J. Geom. Methods Mod. Phys.} \textbf{14} (2017), 1750113 (18 pages).
 \newblock e-print: arXiv:1702.03532.


\bibitem{M1998}
\newblock Mackenzie, K. C. H.,
\newblock {\em Double Lie algebroids and the double of a Lie bialgebroid},
\newblock e-print: arXiv:math/9808081.



\bibitem{Rubtsov}
\newblock  Rubtsov, V.,
\newblock  {\em The cohomology of the Der complex},
\newblock  {Russian Math. Surveys} \textbf{35} (1980), 190--191.



\bibitem{Vitagliano}
\newblock Vitagliano, L.,
\newblock {\em Dirac-Jacobi bundles},
\newblock  {J. Sympl. Geom.} \textbf{18} (2018), in press.
 \newblock e-print: arXiv:1502.05420.

 \bibitem{Vitagliano2}
\newblock Vitagliano, L.,
\newblock {\em $L_\infty$-algebras from multicontact geometry},
 \newblock  {Diff. Geom. Appl.} \textbf{39} (2015), 147--165.
 \newblock e-print: arXiv:1311.2751.

\bibitem{Vitagliano1}
\newblock Vitagliano, L., and A. Wade,
\newblock {\em Holomorphic Jacobi manifolds and complex contact groupoids},
 \newblock e-print: arXiv:1710.0330.

 \bibitem{Vitagliano3}
\newblock Vitagliano, L., and A. Wade,
\newblock {\em Generalized contact bundles},
 \newblock  {C. R. Acad. Sci. Paris, Ser. I} \textbf{354} (2016), 313--317.
 \newblock e-print: arXiv:1507.03973.

 \bibitem{Vitagliano4}
\newblock Vitagliano, L., and J. Schnitzer,
\newblock {\em The local structure of generalized contact bundles},
 \newblock e-print: arXiv:1711.08310.



 \bibitem{V2012}
 \newblock Voronov, T.,
 \newblock {\em Q-manifolds and Mackenzie theory},
 \newblock {Comm. Math. Phys.} \textbf{315} (2012), 279--310.
 \newblock e-print: arXiv:1206.3622.


\bibitem{Wade2000}
\newblock Wade, A.,
\newblock {\em Conformal Dirac structures},
\newblock {Lett. Math. Phys.} \textbf{53} (2000), 331--348.
 \newblock e-print: arXiv:math/0101181.

\bibitem{Wade2004}
\newblock Wade, A.,
\newblock {\em Locally conformal Dirac structures and infinitesimal automorphisms},
\newblock {Commun. Math. Phys.} \textbf{246} (2004), 295--310.

\bibitem{A.Wein}
\newblock Weinstein, A.,
\newblock {\em Omni-Lie algebras},
\newblock {Microlocal analysis of the Schrodinger equation and related Topics (Kyoto, 1999)} \textbf{1176} (2000), 95--102.
\newblock e-print: arXiv:math/9912190.

\bibitem{Zambon2012}
\newblock Zambon, M.,
\newblock {\em $L_\infty$-algebras and higher analogues of Dirac structures and Courant algebroids},
\newblock  {J. Sympl. Geom.} \textbf{10} (2012), 563--599.
\newblock e-print: arXiv:1003.1004.


\end{thebibliography}
\end{document}